\documentclass[12pt]{amsart}
\usepackage[osf,sc]{mathpazo}
\usepackage{amssymb}

\usepackage{geometry}
\usepackage{bbm}
\usepackage{graphicx}
\usepackage{hyperref}
\hypersetup{
    colorlinks=true, 
    linktoc=all,     
    linkcolor=blue,
        citecolor=red,
    filecolor=black,
    urlcolor=blue	
}
\usepackage{enumerate}
\usepackage[inline]{enumitem}
\makeatletter
\newcommand{\inlineitem}[1][]{%
\ifnum\enit@type=\tw@
    {\descriptionlabel{#1}}
  \hspace{\labelsep}%
\else
  \ifnum\enit@type=\z@
       \refstepcounter{\@listctr}\fi
    \quad\@itemlabel\hspace{\labelsep}%
\fi} \makeatother
\newcommand{\gz}{\zeta}

\newcommand{\gth}{\theta}

\newcommand{\gl}{\lambda}

\newcommand{\gp}{\pi}

\newcommand{\gs}{\sigma}

\newcommand{\gt}{\tau}

\newcommand{\gf}{\phi}

\newcommand{\Gd}{\Delta}

\newcommand{\Gom}{\Omega}

\newcommand{\subs}{\subset}

\newcommand{\bs}{\backslash}

\newcommand{\nin}{\notin}

\newcommand{\mbb}{\mathbb}

\newcommand{\mcl}{\mathcal}

\newcommand{\Stab}{\mathrm{Stab}}
\newcommand{\us}{\underset}
\newcommand{\os}{\overset}

\newcommand{\lra}{\longrightarrow}

\newcommand{\N}{\mbb N}
\newcommand{\Z}{\mbb Z}

\newcommand{\La}{\Leftarrow}

\newcommand{\Ra}{\Rightarrow}
\newcommand{\Llra}{\Longleftrightarrow}

\newcommand{\es}{\emptyset}

\newcommand{\equ}[1]{%
\begin{equation*}
#1
\end{equation*}
}
\newcommand{\equa}[1]{%
\begin{equation*}
\begin{aligned}
#1
\end{aligned}
\end{equation*}
}
\newcommand{\equan}[2]{%
\begin{equation}
\label{Eq:#1}
\begin{aligned}
#2
\end{aligned}
\end{equation}
}
\DeclareMathOperator{\Det}{Det}

\DeclareMathOperator{\sgn}{sgn}

\newtheorem{theorem}{Theorem}[section]

\newtheorem{lemma}[theorem]{Lemma}

\theoremstyle{definition}
\newtheorem{defn}[theorem]{Definition}

\theoremstyle{remark}

\newtheorem{remark}[theorem]{Remark}
\numberwithin{equation}{section}
\makeatletter
\def\namedlabel#1#2{\begingroup
   \def\@currentlabel{#2}%
   \label{#1}\endgroup
}
\makeatother

\begin{document}
\title[Point Arrangements and Stratification of Totally Nonzero Grassmannian]{On the Generic Point Arrangements in Euclidean Space and Stratification of the Totally Nonzero Grassmannian}
\author[C.P. Anil Kumar]{Author: C.P. Anil Kumar*}
\address{Post Doctoral Fellow in Mathematics, Harish-Chandra Research Institute, Chhatnag Road, Jhunsi, Prayagraj (Allahabad)-211019, Uttar Pradesh, INDIA
}
\email{akcp1728@gmail.com}
\thanks{*The work is done while the author is a Post Doctoral Fellow at Harish-Chandra Research Institute, Allahabad.}
\date{\sc \today}
\begin{abstract}
	
In this article, for positive integers $n\geq m\geq 1$, the parameter spaces for the isomorphism classes of the generic point arrangements of cardinality $n$, and the antipodal point arrangements of cardinality $2n$ in the Eulidean space $\mbb{R}^m$ are described using the space of totally nonzero Grassmannian $Gr^{tnz}_{mn}(\mbb{R})$. A stratification $\mcl{S}^{tnz}_{mn}(\mbb{R})$ of the totally nonzero Grassmannian $Gr^{tnz}_{mn}(\mbb{R})$ is mentioned and the parameter spaces are respectively expressed as quotients of the space $\mcl{S}^{tnz}_{mn}(\mbb{R})$ of strata under suitable actions of the symmetric group $S_n$ and the semidirect product group $(\mbb{R}^*)^n\rtimes S_n$. The cardinalities of the space $\mcl{S}^{tnz}_{mn}(\mbb{R})$ of strata and of the parameter spaces $S_n\bs \mcl{S}^{tnz}_{mn}(\mbb{R}), ((\mbb{R}^*)^n\rtimes S_n)\bs \mcl{S}^{tnz}_{mn}(\mbb{R})$ are enumerated in dimension $m=2$. Interestingly enough, the enumerated value of the isomorphism classes of the generic point arrangements in the Euclidean plane is expressed in terms of the number theoretic Euler-totient function. The analogous enumeration questions are still open in higher dimensions for $m\geq 3$.
\end{abstract}
\maketitle
\section{\bf{Introduction}}
The stratification of Grassmannians $Gr_{mn}(\mbb{R})$ is an interesting topic of study. There are various types of stratifications of different subsets of Grassmannians. A stratification is a decomposition of the Grassmannian or a subset of the Grassmannian into various strata where each stratum satisfies some nice properties. Some of the decompositions that exist in the literature are:
\begin{enumerate}
\item the decomposition of the Grassmannian $Gr_{mn}(\mbb{R})$ into Schubert cells \linebreak $\{\Gom_{\gl}(\mbb{R})\mid \gl\subseteq (n-m)^m\}$ indexed by partitions $\gl\subseteq (n-m)^m$, 
\item the decomposition of the Grassmannian $Gr_{mn}(\mbb{R})$ into matroid strata \linebreak $\{\mcl{S}_{\mcl{M}}(\mbb{R})\mid \mcl{M} \text{ a realizable matroid of rank }m\}$ also known as Gelfand-Serganova strata labelled by some matroids $\mcl{M}$ called realizable matroids,
\item the decomposition of the totally nonnegative Grassmannian $Gr^{tnz}_{mn}(\mbb{R})$ into totally nonnegative Grassmann cells $\{\mcl{S}^{tnn}_{\mcl{M}}(\mbb{R})=\mcl{S}_{\mcl{M}}(\mbb{R})\cap Gr^{tnn}_{mn}(\mbb{R})\mid \mcl{M} \text{ a realizable matroid of rank }m \text{ such that }\mcl{S}_{\mcl{M}}(\mbb{R})\cap Gr^{tnn}_{mn}(\mbb{R})\neq \es\}$ obtained from the matroid strata $\mcl{S}_{\mcl{M}}(\mbb{R})$ for various $\mcl{M}$.
\end{enumerate}
Moreover in cases (1) and (3), the topological structure of a cell/stratum is homeomorphic to an open ball of appropriate dimension, that is, they are actually cells.
In case (1), $\Gom_{\gl}(\mbb{R}) \cong \mbb{R}^{\mid \gl\mid}$.  In case (3), $\mcl{S}^{tnn}_{\mcl{M}}\cong (\mbb{R}^+)^d$ for some suitable $d\geq 0$. In case (2), the geometric structure of a matroid stratum $\mcl{S}_{\mcl{M}}(\mbb{R})$ can be highly nontrivial as N.~E.~Mn\"{e}v~\cite{MR0970093} has shown that, it can be as complicated as essentially any algebraic variety. 

The enumeration of the number of Schubert cells in the decomposition of $Gr_{mn}(\mbb{R})$ in case (1) is $\binom{n}{m}$  a binomial coefficient and the enumeration of the totally nonnegative Grassmann cells in case (3) is related to Eulerian numbers (refer Section 23 in A.~Postnikov~\cite{arXiv:0609764} and also L.~K.~Williams~\cite{MR2102660}).

In this article, we describe the stratification $\mcl{S}^{tnz}_{mn}(\mbb{R})$ of another subset of the Grassmannian namely, the totally nonzero Grassmannian $Gr^{tnz}_{mn}(\mbb{R})$. We also enumerate the number of strata in $\mcl{S}^{tnz}_{2n}(\mbb{R})$, that is, for $m=2$ to be equal to $2^{n-2}(n-1)!$ for $n\geq 2$. The enumeration question is still open for dimensions $m>2$. Also we relate the totally nonzero Grassmannian $Gr^{tnz}_{mn}(\mbb{R})$ to the generic point arrangements of cardinality $n$ and the antipodal point arrangements of cardinality $2n$ in Euclidean space $\mbb{R}^m$ and describe the parameter spaces for the isomorphism classes of the generic point arrangements and the antipodal point arrangements via certain group actions on the space $\mcl{S}^{tnz}_{mn}(\mbb{R})$ of strata. Again we enumerate the parameter spaces $S_n\bs \mcl{S}^{tnz}_{2n}(\mbb{R}), ((\mbb{R}^*)^n\rtimes S_n)\bs \mcl{S}^{tnz}_{2n}(\mbb{R})$ for $m=2$. The question of the enumeration of the set $S_n\bs \mcl{S}^{tnz}_{2n}(\mbb{R})$ has an interesting answer. The definitions and details are mentioned in the next and later sections.  

\section{\bf{Point Arrangements in Euclidean Spaces}}
\begin{defn}[A Point Arrangement, An Generic Point Arrangement, An Antipodal Point Arrangement]
Let $n\geq m \geq 1$ be two positive integers. Any finite subset $S\subs \mbb{R}^m$ is a said to be a point arrangement. A point arrangement $S=\{v_1,v_2,\cdots,v_n\}$ $\subs \mbb{R}^m$  is said to be generic if any subset $I\subs S$ of cardinality at most $m$ is a linearly independent set. A point arrangement $S=\{v_1,v_2,\cdots,v_{2n}\}\subs \mbb{R}^m$ is said to be an antipodal point arrangement if for every $1\leq i\leq 2n$ there exists unique $1\leq j=\gs(i)\leq 2n$ such that $v_i=tv_j$ for some real number $t<0$ and the subset $I$ of cardinality $n\geq m$ consisting of one representative from each of the sets $\{v_i,v_{\gs(i)}\}$ is a generic point arrangement in $\mbb{R}^m$.  Conventionally, we choose $v_{\gs(i)}=-v_i$ for every $1\leq i\leq 2n$ and we call the subset $\{v_i,v_{\gs(i)}\}$ a set consisting of an antipodal pair. 
\end{defn}
\begin{defn}[Isomorphic Generic Point Arrangements, Isomorphic Antipodal Point Arrangements]
We say two generic point arrangements $S_1,S_2$ in $\mbb{R}^m$ are isomorphic if there is a set bijection $\gs:S_1\lra S_2$ such that for any subset $C=\{v_{i_1},v_{i_2},\cdots,v_{i_m},v_{i_{m+1}}\}\subseteq S_1$ of cardinality $(m+1)$ if $v_{i_{m+1}}=\us{j=1}{\os{m}{\sum}}a_jv_{i_j}$ and if $\gs(v_{i_{m+1}})=\us{j=1}{\os{m}{\sum}}b_j\gs(v_{i_j})$ then for each $1\leq j\leq m,\ $  $a_jb_j>0$, that is, $\sgn(a_j)=\sgn(b_j)\neq 0$ where $\sgn$ is the signum function.

We say two conventional antipodal point arrangements $S_1,S_2$ are isomorphic if there exists a bijection $\gs:S_1\lra S_2$ such that $\gs(-v)=-\gs(v)$ for every $v\in S_1$ and if for every subset $C=\{v_{i_1},v_{i_2},\cdots,v_{i_m},v_{i_{m+1}}\}\subseteq S_1$ of cardinality $(m+1)$ not containing any antipodal pair, if $v_{i_{m+1}}=\us{j=1}{\os{m}{\sum}}a_jv_{i_j}$ and if $\gs(v_{i_{m+1}})=\us{j=1}{\os{m}{\sum}}b_j\gs(v_{i_j})$ then $a_j>0$ for all $1\leq j\leq m$ if and only if  $b_j>0$ for all $1\leq j\leq m$.
\end{defn}
\begin{remark}
Let $S$ be a generic point arrangement in $\mbb{R}^m$. Let $A\in GL_m(\mbb{R})$. Then $T=AS=\{Av\mid v\in S\}$ is also a generic point arrangement in $\mbb{R}^m$. Moreover $S$ is isomorphic to $T$. Similarly if $S$ is an antipodal point arrangement in $\mbb{R}^m$ then $T$ is also an antipodal point arrangement in $\mbb{R}^m$ and $S$ is isomorphic to $T$.
\end{remark}
\section{\bf{Grassmannians}}
For $n\geq m\geq 1$, the Grassmannian $Gr_{mn}(\mbb{R})$ is the collection of $m$-dimensional subspaces $V\subseteq \mbb{R}^n$. It can be presented as the quotient \equ{Gr_{mn}(\mbb{R})=GL_m(\mbb{R})\bs Mat^*_{mn}(\mbb{R}),} where $Mat^*_{mn}(\mbb{R})$ is the space of $(m\times n)$-matrices of rank $m$. Here we assume that the subspace $V$ associated with a $(m\times n)$-matrix $M$ is spanned by the row vectors of $M$.

\subsection{\bf{Pl\"{u}cker coordinates}} 
For a $(m\times n)$-matrix $M$ and a $m$-element subset $I\subseteq [n]=\{1,2,\cdots,n\}$, let $M_I$ denote the $(m\times m)$-submatrix of $M$ in the column set $I$, and let $\Gd_I(M):=\Det(M_I)$ denote the maximal minor of $M$. If we multiply $M$ by $A\in GL_m(\mbb{R})$ on the left, all minors $\Gd_I(M)$ are rescaled by the same factor $\Det(A)$. If $M=(m_{ij})$ is in $I$-echelon form then $M_I=Id_m$ and $m_{ij}=+\Gd_{I\bs \{i\}\cup\{j\}}$ or $m_{ij}=-\Gd_{I\bs \{i\}\cup\{j\}}$ and the sign can be determined. Thus the $(\Gd_I)_{I\in \binom{[n]}{m}}$ form projective coordinates of the Grassmannian $Gr_{mn}(\mbb{R})$ called the Pl\"{u}cker coordinates and the map $M\lra (\Gd_I)_{I\in \binom{[n]}{m}}$ induces the Pl\"{u}cker embedding $Gr_{mn}(\mbb{R}) \hookrightarrow \mbb{RP}^{\binom{n}{m}-1}$ of the Grassmannian into the projective space. The image of the Grassmannian $Gr_{mn}(\mbb{R})$ under the Pl\"{u}cker embedding is the algebraic subvariety in $\mbb{RP}^{\binom{n}{m}-1}$ given by the Grassmann-Pl\"{u}cker relations:
\equ{\Gd_{(i_1,i_2,\cdots,i_m)}.\Gd_{(j_1,j_2,\cdots,j_m)}=\us{s=1}{\os{m}{\sum}}\Gd_{(j_s,i_2,\cdots,i_m)}.\Gd_{(j_1,j_2,\cdots,j_{s-1},i_1,j_{s+1},j_{s+2},\cdots,j_m)},}
for any $i_1,i_2,\cdots,i_m,j_1,j_2,\cdots,j_m\in [n]$. Here we assume that $\Gd_{(i_1,i_2,\cdots,i_m)}$ (labelled by an ordered sequence rather than a subset) equals to $\Gd_{\{i_1,i_2,\cdots,i_m\}}$ if $i_1<i_2<\cdots<i_m$ and $\Gd_{(i_1,i_2,\cdots,i_m)}=(-1)^{sign(w)}\Gd_{(i_{w(1)},i_{w(2)},\cdots,i_{w(m)})}$ for all $w\in S_m$.

\subsection{\bf{The Totally Nonzero Grassmannian}}

\begin{defn}
	Let us define the totally nonzero Grassmannian $Gr^{tnz}_{mn}(\mbb{R})\subs Gr_{mn}(\mbb{R})$ as the quotient $Gr^{tnz}_{mn}(\mbb{R})=GL_m(\mbb{R})\bs Mat^{tnz}_{mn}(\mbb{R})$ where $Mat^{tnz}_{mn}(\mbb{R})$ is the set of real $(m\times n)$-matrices $M$ of rank $m$ with all maximal minors $\Gd_I(M)\neq 0$.
	If $M=(v_1,v_2,\cdots,v_n)$ where $v_i$ is the $i^{th}$-column of the matrix $M$ for $1\leq i\leq n$ then the set $S=\{v_i\mid 1\leq i\leq n\}$ is a generic point arrangement. Moreover if $N=(w_1,w_2,\cdots,w_n)$ is another matrix representing the same element $V$ of the Grassmannian then there exists a matrix $A\in GL_m(\mbb{R})$ such that $AM=N$. So the set $T=\{w_1,w_2,\cdots,w_n\}=AS$ is a generic point arrangement isomorphic to $S$. Hence each element $V=GL_m(\mbb{R}).M\in Gr_{mn}(\mbb{R})$ represents an isomorphism class of a generic point arrangement.
\end{defn}
\subsection{\bf{Stratification of the Totally Nonzero Grassmannian}}
\begin{defn}
	Let $\mcl{C}\subseteq \binom{[n]}{m}$ be a certain collection of $m$-subsets of $\{1,2,\cdots,n\}$. Define the stratum $\mcl{S}^{tnz}_{\mcl{C}}(\mbb{R})=\{GL_m(\mbb{R}).M\in Gr^{tnz}_{mn}(\mbb{R}) \mid \text{ either } \Gd_I(M)>0 \text{ for all }$ $I\in \mcl{C} \text{ and }\Gd_I(M)$ $<0 \text{ for all } I\nin \mcl{C} \text{ or }
	\Gd_I(M)<0$ for all $I\in \mcl{C} \text{ and }\Gd_I(M)>0$ for all $I\nin \mcl{C} \}$. Note $\mcl{S}^{tnz}_{\mcl{C}}(\mbb{R})=\mcl{S}^{tnz}_{\binom{[n]}{m}\bs \mcl{C}}(\mbb{R})$. Let $\mcl{S}^{tnz}_{mn}(\mbb{R})=\{\mcl{S}^{tnz}_{\mcl{C}}(\mbb{R})\mid \mcl{C}\subseteq \binom{[n]}{m},\mcl{S}^{tnz}_{\mcl{C}}(\mbb{R})\neq \es\}$ denote the collection of nonempty strata. 
\end{defn}
\begin{remark}
	Using Pl\"{u}cker relations, we can produce a collection $\mcl{C}$ whose stratum $\mcl{S}^{tnz}_{\mcl{C}}(\mbb{R})$ is empty.
\end{remark}
\begin{remark}
	The description of the totally positive Grassmannian and the totally nonnegative Grassmannian and their stratifications are given in A.~Postnikov~\cite{arXiv:0609764}.
\end{remark}
\subsection{\bf{Some Group Actions on Strata}}
\begin{defn}
	We define the action of $(\mbb{R}^*)^n$ on $\mcl{S}^{tnz}_{mn}(\mbb{R})$. Let $\mcl{S}^{tnz}_{\mcl{C}}(\mbb{R})$ be a non-empty stratum, that is, $\mcl{S}^{tnz}_{\mcl{C}}(\mbb{R})\in \mcl{S}^{tnz}_{mn}(\mbb{R})$ for some $\mcl{C}\in \binom{[n]}{m}$. Let $(t_1,t_2,\cdots,t_n)\in (\mbb{R}^*)^n$ and $GL_m(\mbb{R}).M\in \mcl{S}^{tnz}_{\mcl{C}}(\mbb{R})$ where $M=(v_1,v_2,\cdots,v_n)$ with $v_i\in \mbb{R}^m$ a column vector of $M$ for $1\leq i\leq n$. Define $N=(t_1v_1,t_2v_2,\cdots,t_nv_n)$ and define the action $(t_1,t_2,\cdots,t_n)\bullet GL_m(\mbb{R}).M=GL_m(\mbb{R}).N$. We observe that $\Gd_I(N)=\bigg(\us{i\in I}{\prod}t_i\bigg)\Gd_I(M)$. Now define the collection $\mcl{D}=\big\{I\in \binom{[n]}{m}\text{ such that }\Gd_I(N)>0\big\}$. Then the action of $(\mbb{R}^*)^n$ on the strata is given by \equ{(t_1,t_2,\cdots,t_n)\bullet\mcl{S}^{tnz}_{\mcl{C}}(\mbb{R})=(t_1,t_2,\cdots,t_n).\mcl{S}^{tnz}_{\binom{[n]}{m}\bs \mcl{C}}(\mbb{R})=\mcl{S}^{tnz}_{\mcl{D}}(\mbb{R})=\mcl{S}^{tnz}_{\binom{[n]}{m}\bs \mcl{D}}(\mbb{R}).}
\end{defn}
\begin{defn}
	We define the action of $(\mbb{R}^+)^n$ on $\mcl{S}^{tnz}_{mn}(\mbb{R})$. This action is the restriction of the action of $(\mbb{R}^*)^n$. Here in fact we observe that $(t_1,t_2,\cdots,t_n)\bullet\mcl{S}^{tnz}_{\mcl{C}}(\mbb{R})=\mcl{S}^{tnz}_{\mcl{C}}(\mbb{R})$. Every stratum is a fixed point for this action. So on the strata this action is trivial.
\end{defn}
\begin{defn}
	We define the action of $S_n$ on $\mcl{S}^{tnz}_{mn}(\mbb{R})$. Let $P$ be an $(n\times n)$-permutation matrix in $S_n$. Let $\mcl{S}^{tnz}_{\mcl{C}}(\mbb{R})$ be a non-empty stratum and $GL_m(\mbb{R}).M\in \mcl{S}^{tnz}_{\mcl{C}}(\mbb{R})$. Then define $P\bullet (GL_m(\mbb{R}).M)=GL_m(\mbb{R}).MP^{-1}$. If the permutation matrix $P$ is denoted by $\gs:[n]\lra [n]$ then for $1\leq i\leq n, P^{-1}(e^n_i)^t=(e^n_{\gs^{-1}(i)})^t$ where $e_i^n$ is the standard $n$-dimensional row vector which has entries $1$ at the $i^{th}$ place and zeroes elsewhere and $t$ stands for transpose. If $M=(v_1,v_2,\cdots,v_n)$ then $MP^{-1}=(v_{\gs^{-1}(1)},v_{\gs^{-1}(2)},\cdots,v_{\gs^{-1}(n)})$. This action gives rise to a well-defined action on the strata, that is, if $GL_m(\mbb{R}).M$,  $GL_m(\mbb{R}).N$ belong to the same stratum $\mcl{S}^{tnz}_{\mcl{C}}(\mbb{R})$ then $GL_m(\mbb{R}).MP^{-1}, GL_m(\mbb{R}).NP^{-1}$ belong to the same stratum. So
	\equ{\gs\bullet \mcl{S}^{tnz}_{\mcl{C}}(\mbb{R})=\mcl{S}^{tnz}_{\mcl{D}}(\mbb{R}) }
	where $GL_m(\mbb{R}).MP^{-1}\in \mcl{S}^{tnz}_{\mcl{D}}(\mbb{R})$.
\end{defn}
\begin{defn}
	We define the semi-direct product group $(\mbb{R}^*)^n\rtimes S_n$ as follows.
	Let $(s_1,s_2,\cdots,s_n),(t_1,t_2,\cdots,t_n)\in (\mbb{R}^*)^n$ and $\gs,\gt\in S_n$. Define \equ{((s_1,s_2,\cdots,s_n),\gt).((t_1,t_2,\cdots,t_n),\gs)=((s_{\gs(1)}t_1,s_{\gs(2)}t_2,\cdots,s_{\gs(n)}t_n),\gt\gs).}
	Let $\mcl{S}^{tnz}_{\mcl{C}}(\mbb{R})$ be a non-empty stratum and $GL_m(\mbb{R}).M\in \mcl{S}^{tnz}_{\mcl{C}}(\mbb{R})$. Now define the action of $(\mbb{R}^*)^n\rtimes S_n$ on the element $GL_m(\mbb{R}).M \in Gr_{mn}^{tnz}(\mbb{R})$ as 
	\equ{((t_1,t_2,\cdots,t_n),\gs)\bullet (GL_m(\mbb{R}).M)=GL_m(\mbb{R}).N} with 
	$N=(t_{\gs^{-1}(1)}v_{\gs^{-1}(1)},t_{\gs^{-1}(2)}v_{\gs^{-1}(2)},\cdots,t_{\gs^{-1}(n)}v_{\gs^{-1}(n)})$ where $M=(v_1,v_2,\cdots,$ $v_n)$.
	This action gives rise to an action on the space $\mcl{S}^{tnz}_{mn}(\mbb{R})$ of strata as \equ{((t_1,t_2,\cdots,t_n),\gs)\bullet \mcl{S}^{tnz}_{\mcl{C}}(\mbb{R})=\mcl{S}^{tnz}_{\mcl{D}}(\mbb{R})}
	where $GL_m(\mbb{R}).N\in \mcl{S}^{tnz}_{\mcl{D}}(\mbb{R})$.
\end{defn}
\begin{theorem}
\label{theorem:SNAction}
Consider the action of $S_n$ on the space $\mcl{S}^{tnz}_{mn}(\mbb{R})$ of strata. Let $\mcl{C},\mcl{D}$ be two collections of $m$-subsets of $[n]$ such that $\mcl{S}^{tnz}_{\mcl{C}}(\mbb{R}),\mcl{S}^{tnz}_{\mcl{D}}(\mbb{R})$ are non-empty. Then $\mcl{S}^{tnz}_{\mcl{C}}(\mbb{R})$ and $\mcl{S}^{tnz}_{\mcl{D}}(\mbb{R})$ are in the same orbit for the action of $S_n$ if and only if for any two matrices $M=(v_1,v_2,\cdots,v_n),N=(w_1,w_2,\cdots,w_n)\in Mat^{tnz}_{mn}(\mbb{R})$ such that $GL_m(\mbb{R}).M\in \mcl{S}^{tnz}_{\mcl{C}}(\mbb{R}), GL_m(\mbb{R}).N\in \mcl{S}^{tnz}_{\mcl{D}}(\mbb{R})$ the sets $S=\{v_1,v_2,\cdots,v_n\}$ and $T=\{w_1,w_2,\cdots,$ $w_n\}$ are isomorphic generic point arrangements.
\end{theorem}
\begin{proof}
($\La$) Let $S$ and $T$ be two isomorphic arrangements and $\gs:T\lra S$ be an isomorphism. Let $\gs(w_i)=v_{\gs(i)},1\leq i\leq n$. Let $M=(v_1,v_2,\cdots,v_n)$ and $N_1=(w_1,w_2,\cdots,w_n)P^{-1}=(w_{\gs^{-1}(1)},w_{\gs^{-1}(2)},\cdots,w_{\gs^{-1}(n)})$ where $P$ is a permutation matrix such that $P^{-1}(e^n_{i})^t=(e^n_{\gs^{-1}(i)})^t$. Then we prove that $GL_m(\mbb{R}).N_1\in \mcl{S}^{tnz}_{\mcl{C}}(\mbb{R})$ which is the stratum that contains $GL_m(\mbb{R}).M$.
First assume without loss of generality that $v_i=(e^m_i)^t=w_{\gs^{-1}(i)},1\leq i\leq n$. Then $\Det(v_1,v_2,\cdots,v_m)=\Det(w_{\gs^{-1}(1)},w_{\gs^{-1}(2)},\cdots,w_{\gs^{-1}(m)})=1$. Moreover for $i>m$ we have if $v_i=(x_{i1},x_{i2},\cdots,x_{im})^t$ and $w_{\gs^{-1}(i)}=(y_{i1},y_{i2},\cdots,y_{im})^t$ then $\sgn(x_{ij})=\sgn(y_{ij})$ for $1\leq j\leq m<i\leq n$. This follows because $\gs^{-1}:v_i\lra w_{\gs^{-1}(i)},1\leq i\leq n$ is an isomorphism. Hence $\sgn(\Det(v_{i_1},v_{i_2},\cdots,v_{i_{m-1}},v_i))=\sgn(\Det(w_{\gs^{-1}(i_1)},v_{\gs^{-1}(i_2)},\cdots,$ $v_{\gs^{-1}(i_{m-1})},w_{\gs^{-1}(i)}))$ for any $1\leq i_1<i_2<\cdots<i_{m-1}\leq m<i\leq n$. So the signs of the determinants agree for any $m$-subset $I$ of $[n]$ such  that $\mid I\cap [m]\mid =m-1$ assuming that the sign of the determinant agrees for the $m$-subset $[m]\subseteq [n]$. Now the proof can be extended to all $m$ subsets of $[n]$. So the elements $GL_m(\mbb{R}).N_1,GL_m(\mbb{R}).M$ are in the same stratum. Hence the elements  $GL_m(\mbb{R}).N,GL_m(\mbb{R}).M$ are in the same $S_n$-orbit  implying that the strata $\mcl{S}^{tnz}_{\mcl{C}}(\mbb{R}),\mcl{S}^{tnz}_{\mcl{D}}(\mbb{R})$ are in the same $S_n$-orbit..

$(\Ra)$ Now assume that the stratum $\mcl{S}^{tnz}_{\mcl{C}}(\mbb{R}),\mcl{S}^{tnz}_{\mcl{D}}(\mbb{R})$  are in the same $S_n$-orbit. Let $\gs\bullet \mcl{S}^{tnz}_{\mcl{C}}(\mbb{R})= \mcl{S}^{tnz}_{\mcl{D}}(\mbb{R})$. Let $GL_m(\mbb{R}).M\in \mcl{S}^{tnz}_{\mcl{C}}(\mbb{R}),GL_m(\mbb{R}).N\in \mcl{S}^{tnz}_{\mcl{D}}(\mbb{R})$ where $M=(v_1,v_2,\cdots,v_n),N=(w_1,w_2,\cdots,w_n), N\gs^{-1}=(w_{\gs^{-1}(1)},w_{\gs^{-1}(2)},\cdots,w_{\gs^{-1}(n)})$. If the signs of the coordinates $(\Gd_I(M))_{I\in \binom{[n]}{m}}$ and $(\Gd_I(N\gs^{-1}))_{I\in \binom{[n]}{m}}$ agree uniformly then we do not need to reflect $N$. Otherwise if the signs of the coordinates uniformly disagree then we reflect $N$ by using a reflection $R\in GL_m(\mbb{R})$ and consider $RN$ which does not change the isomorphism class of $N$. Now the previous proof can be traced back to obtain an isomorphism $\gs:T\lra S$. So the generic point arrangements $T$ and $S$ are isomorphic.
\end{proof}
\begin{theorem}
\label{theorem:SDGAction}
Consider the action of $(\mbb{R}^*)^n\rtimes S_n$ the space $\mcl{S}^{tnz}_{mn}(\mbb{R})$ of strata. Let $\mcl{C},\mcl{D}$ be two collections of $m$-subsets of $[n]$ such that $\mcl{S}^{tnz}_{\mcl{C}}(\mbb{R}),\mcl{S}^{tnz}_{\mcl{D}}(\mbb{R})$ are non-empty. Then $\mcl{S}^{tnz}_{\mcl{C}}(\mbb{R})$ and $\mcl{S}^{tnz}_{\mcl{D}}(\mbb{R})$ are in the same orbit for the action of $(\mbb{R}^*)^n\rtimes S_n$ if and only if for any two matrices $M=(v_1,v_2,\cdots,v_n),N=(w_1,w_2,\cdots,w_n)\in Mat^{tnz}_{mn}(\mbb{R})$ such that $GL_m(\mbb{R}).M\in \mcl{S}^{tnz}_{\mcl{C}}(\mbb{R}), GL_m(\mbb{R}).N\in \mcl{S}^{tnz}_{\mcl{D}}(\mbb{R})$ the sets $S=\{\pm v_1,\pm v_2,\cdots,\pm v_n\}$ and $T=\{\pm w_1,\pm w_2,\cdots,\pm w_n\}$ are isomorphic antipodal point arrangements.
\end{theorem}
\begin{proof}
The proof is similar to the proof of Theorem~\ref{theorem:SNAction} except here the group which is acting on the strata is $(\mbb{R}^*)^n\rtimes S_n$.  	
\end{proof}
\begin{theorem}
\label{theorem:Enumeration}
Let $n\geq m\geq 1$ be two positive integers.
\begin{enumerate}
\item  The isomorphism classes of the generic point arrangements of cardinality $n$ in the Euclidean space $\mbb{R}^m$ are in bijection with the  orbits in $S_n\bs\mcl{S}^{tnz}_{mn}(\mbb{R})$ under the action of $S_n$ on the space $\mcl{S}^{tnz}_{mn}(\mbb{R})$ of strata.
\item The isomorphism classes of the antipodal point arrangements of cardinality $2n$ in the Euclidean space $\mbb{R}^m$ are in bijection with the orbits in $((\mbb{R}^*)^n\rtimes S_n)\bs\mcl{S}^{tnz}_{mn}(\mbb{R})$ under the action of $(\mbb{R}^*)^n\rtimes S_n$ on the space $\mcl{S}^{tnz}_{mn}(\mbb{R})$ of strata.
\end{enumerate}
\end{theorem}
\begin{proof}
The theorem follows from Theorem~\ref{theorem:SNAction} and Theorem~\ref{theorem:SDGAction}.
\end{proof}
\section{\bf{On the Strata in Dimension Two}}
First we represent any generic point arrangement in $\mbb{R}^2$ combinatorially by capturing all the required geometric data combinatorially.

\subsection{Combinatorial Representation of a Generic Point Arrangement in $\mbb{R}^2$}
\label{sec:CombRepGPA}
Let $S=\{v_1,v_2,\cdots,v_n\}$ be a generic point arrangement in $\mbb{R}^2$. Assume the 
Let $v_i=(x_i,y_i)\in \mbb{R}^2$. Note $v_i\neq 0\in \mbb{R}^2$. Let the angle the line $L_i=\{tv_i\mid t\in \mbb{R}\}$ makes with respect to the positive $X$-axis be $\gth_i$. We note that the lines $L_i,1\leq i\leq n$ are all distinct.  Assume after renumbering the subscripts of the elements of $S$ we have $0\leq \gth_1<\gth_2<\cdots<\gth_n<\gp$. A combinatorial representation of the set $S$ is an element $(a_1,a_2,\cdots,a_n)$ in the set $\{+1,-1\}\times \{+2,-2\}\times \cdots \times \{+n,-n\}$ defined as follows. If $y_1=0$ then $a_1=\sgn(x_1)$ else $a_1=\sgn(y_1)$. We define for $i>1, a_i=\sgn(y_i)i$. We note that for $i>1, y_i\neq 0$. 

\begin{lemma}
If $S=\{v_1,v_2,\cdots,v_n\}$ is a generic point arrangement and the line $L_i=\{tv_i\mid t\in \mbb{R}\}$ makes an angle $\gth_i$ with respect to the positive $X$-axis with $0\leq \gth_1<\gth_2<\cdots<\gth_n<\gp$. Let $(a_1,a_2,\cdots,a_n)\in \us{i=1}{\os{n}{\prod}}\{+i,-i\}$ be the combinatorial representation of the set $S$. Then we have for $1\leq i\neq j\leq n, \sgn(\Det(v_i,v_j))=\sgn(a_ia_j(j-i))$.
\end{lemma}
\begin{proof}
The proof is an observation obtained by considering all the various cases.	
\end{proof}
\begin{defn}
For a matrix $M=(v_1,v_2,\cdots,v_n)\in Mat^{tnz}_{2n}(\mbb{R})$ we define the orientation sign matrix $O_M=[\sgn(\Det(v_i,v_j))]_{1\leq i,j\leq n}$. Let $\mcl{O}_{n}(\mbb{R})=\{O_M\mid M\in Mat^{tnz}_{2n}(\mbb{R}) \}$ be space of orientation sign matrices.
\end{defn}
\begin{remark}
 If $M=(v_1,v_2,\cdots,v_n)\in Mat^{tnz}_{2n}(\mbb{R})$ and $P$ is a permutation matrix associated to the permutation $\gs:[n]\lra [n]$ then $P^{-1}(e^n_i)^t=(e^n_{\gs^{-1}(i)})^t$ and for the matrix $N=(v_{\gs^{-1}(1)},v_{\gs^{-1}(2)},\cdots,v_{\gs^{-1}(n)})=MP^{-1}$ the orientation sign matrix $O_N=[\sgn(\Det(v_{\gs^{-1}(i)},v_{\gs^{-1}(j)}))]_{1\leq i,j\leq n}=PO_MP^{-1}$ (also denoted by $\gs O_M\gs^{-1}$).
\end{remark}
\begin{remark}
If $M=(v_1,v_2,\cdots,v_n)\in Mat^{tnz}_{2n}(\mbb{R})$ then the plucker coordinates $(\Gd_I(M))_{I=\{i,j\}\in \binom{[n]}{2}}$ of $GL_2(\mbb{R}).M\in Gr^{tnz}_{2n}(\mbb{R})$ satisfy that its associated sign vector is given either by the strictly upper triangular entries of the orientation sign matrix $O_M$ or it is given by the strictly lower triangular entries of $O_M$. Moreover we note that if $\gs:[n]\lra [n]$ is the permutation which takes $i$ to $n-i+1, 1\leq i\leq n$ then $-O_M=O_M^t=\gs O_M\gs^{-1}=O_N$ where $N=(v_n,v_{n-1},\cdots,v_2,v_1)$.  
\end{remark}
\begin{remark}
It is therefore easy to observe that the group $\frac{\Z}{2\Z}=\{+1,-1\}$ acts on the set $\mcl{O}_{n}(\mbb{R})$ of orientation sign matrices as $-1\bullet O_M=O_N$ where $M=(v_1,v_2,\cdots,v_n)$, $N=(v_n,v_{n-1},\cdots,v_2,v_1)\in Mat^{tnz}_{2n}(\mbb{R})$ and the set $\mcl{S}^{tnz}_{2n}(\mbb{R})$ is in bijection with the orbit space $\frac{\Z}{2\Z}\bs \mcl{O}_{n}(\mbb{R})$.
\end{remark}

\subsection{Combinatorial Representation of an Element in $Mat^{tnz}_{2n}(\mbb{R})$}
\label{sec:CombRepElement}
Now we extend the definition of a combinatorial representation of a generic point arrangement to a combinatorial representation of an element in $Mat^{tnz}_{2n}(\mbb{R})$.
Let $M=(v_1,v_2,\cdots,v_n)\in Mat^{tnz}_{2n}(\mbb{R})$. Let $S=\{v_i\mid 1\leq i\leq n\}$. Let the combinatorial representation of $S$ be given by the element  $(a_1,a_2,\cdots,a_n)\in \us{i=1}{\os{n}{\prod}}\{+i,-i\}$. Let the line $L_i=\{tv_i\mid t\in \mbb{R}\}$ make an angle $\gth_{\gt(i)}$ with respect to the positive $X$-axis for a permutation $\gt\in S_n$ such that $0\leq \gth_{1}<\gth_{2}<\cdots<\gth_{n}<\gp$.
Then the combinatorial representation of $M$ is a permutation of the coordinates of the element given by $(a_{\gt(1)},a_{\gt(2)},\cdots,a_{\gt(n)})$. This motivates the definition of signed permutation group given in Definition~\ref{defn:SPG}.

\subsection{Enumeration of Strata in Dimension Two}
In this section we enumerate the number of strata in dimension two that is the cardinality of the set $\mcl{S}^{tnz}_{2n}(\mbb{R})$.
\begin{defn}
\label{defn:SPG}
Let $P_n=\{(a_1,a_2,\cdots,a_n)\in \{\pm 1,\pm 2,\cdots,\pm n\}^n\mid \text{ there exists }\gp\in S_n \text{ with }\gp(i)=\sgn(a_i) a_i\}$, be the signed permutation group. We have $P_n\cong S_n\ltimes (\Z/2\Z)^n$ with the isomorphism being \equ{(a_1,a_2,\cdots,a_n)\lra \big(\gp,(\sgn(a_1),\sgn(a_2),\cdots,\sgn(a_n))\big)} where $\gp(i)=\sgn(a_i)a_i$. Let $p=(a_1,a_2,\cdots,$ $a_n)\in P_n$. The signed permutation matrix associated to $p$ is $M=(v_1,v_2,\cdots,v_n)\in GL_n(\mbb{R})$ where $v_i=\sgn(a_i)e^t_{\gp(i)}$ and $\sgn(a_i) a_i =\gp(i)$. The group multiplication in $P_n$ is given by the corresponding matrix multiplication as a subgroup of $GL_n(\mbb{R})$. Let $K_n\cong \frac{\Z}{2n\Z}$ be the cyclic subgroup of order $2n$ generated by the element $(-n,1,\cdots,(n-2),(n-1))\in P_n$.
\end{defn}
\begin{theorem}
\begin{enumerate}
\item The set $\{O_M\mid M\in Mat^{tnz}_{2n}(\mbb{R})\}$ of orientation sign matrices is in bijection with the rights cosets of $K_n$ in $P_n$ and its cardinality is $2^{n-1}(n-1)!$ for $n\geq 2$.
\item The set $\mcl{S}^{tnz}_{2n}(\mbb{R})$ of strata in dimension two which is in bijection with the orbit space $\frac{\Z}{2\Z}\bs \mcl{O}_{n}(\mbb{R})$ has cardinality $2^{n-2}(n-1)!$ for $n\geq 2$.
\end{enumerate}
\end{theorem}
\begin{proof}
Let $M=(v_1,v_2,\cdots,v_n)\in Mat^{tnz}_{2n}(\mbb{R})$ be such that the combinatorial representation of the generic point arrangement $S=\{v_i\mid 1\leq i\leq n\}$ belongs to $K_n$. Then there exists a permutation $P_{\gs}\in S_n$ such that $MP_{\gs}^{-1}=(v_{\gs^{-1}}(1),v_{\gs^{-1}}(2),\cdots,v_{\gs^{-1}}(n))$ and the orientation sign matrix $O_{MP^{-1}_{\gs}}=P_{\gs}O_MP^{-1}_{\gs}$ is the standard matrix $O_n$ given by $O_{MP^{-1}_{\gs}}=O_n=[a_{ij}]_{1\leq i,j\leq n}$ with $a_{ij}=1$ if $i<j$, $a_{ij}=-1$ if $i>j$ and $a_{ij}=0$ if $i=j$. Conversely it is also clear from the arrangement of points $v_i,1\leq i \leq n$ in the plane that, if the orientation matrix of $O_{MP^{-1}_{\gs}}$ is the standard matrix $O_n$ for some permutation $P_{\gs}\in S_n$ then the combinatorial representation of $S$ belongs to $K_n$.

Now we observe that the combinatorial representation itself, of $M$ given by $(a_1,a_2,\cdots,$ $a_n)$ in $P_n$, belongs to $K_n$ if and only if the orientation sign matrix $O_M$ is the standard matrix. So the combinatorial representation of $M$ belongs to the coset $K_np^{-1}$ for $p\in P_n$ if and only if the orientation sign matrix $O_M=PO_nP^{-1}$ where $P$ is the associated signed permutation matrix of $p$ and $O_n$ is the standard orientation sign matrix. Also any orientation sign matrix is a conjugate of $O_n$ by some element $p\in P_n$. Hence $(1)$ follows and the number of right cosets is $\frac{\mid P_n\mid}{\mid K_n\mid} = \frac{2^nn!}{2n}=2^{n-1}(n-1)!$.

Now (2) follows since the orientation sign matrices $O_M$ and $O_M^t=-O_M=O_{RM}=O_N$ correspond to the same stratum using a reflection $R$ of the plane. Here  $N=(v_n,v_{n-1},\cdots,v_1)$. Hence $\mid \frac{\Z}{2\Z}\bs \mcl{O}_{n}(\mbb{R}) \mid =2^{n-2}(n-1)!$. This proves the theorem.
\end{proof} 
\subsection{Enumeration of the Isomorphism Classes of the Antipodal Point Arrangements in the Plane}
The theorem is stated as follows.
\begin{theorem}
\label{theorem:Trasitive}
The group $((\mbb{R}^*)^n\rtimes S_n)$ acts transitively on $\mcl{S}^{tnz}_{2n}(\mbb{R})$. Hence
the orbit space $((\mbb{R}^*)^n\rtimes S_n)\bs\mcl{S}^{tnz}_{2n}(\mbb{R})$ is a singleton set for $n\geq 2$.
\end{theorem}
\begin{proof}
It is clear geometrically that, if there are $n$-antipodal pairs in $\mbb{R}^2$ forming an antipodal point arrangement in the plane then they can be arranged on lines $L_1,L_2,\cdots,L_n$ passing through origin making angles $\gth_i,1\leq i\leq n$ with respect to the positive $X$-axis such that $0\leq \gth_1<\gth_2<\cdots<\gth_n<\gp$. Hence the group $((\mbb{R}^*)^n\rtimes S_n)$ acts transitively on $\mcl{S}^{tnz}_{2n}(\mbb{R})$. Also group theoretically, we see that the signed permutation group $P_n$ acts transitively on the set $\{O_M\mid M\in Mat^{tnz}_{2n}(\mbb{R})\}$ of orientation sign matrices by conjugation. Now the theorem follows. 
\end{proof}
\subsection{Enumeration of the Isomorphism Classes of the Generic Point Arrangements in the Plane}

In this section we enumerate the cardinality of the space of orbits $S_n\bs\mcl{S}^{tnz}_{2n}(\mbb{R})$ under the action of $S_n$ on the space of strata $\mcl{S}^{tnz}_{2n}(\mbb{R})$.

\begin{theorem}
\label{theorem:FixPoints}
Let $n$ be a positive integer and $\gz$ be a generator of the cyclic group $\frac {\Z}{2n\Z}$ of order $2n$. Consider the following action of $\frac{\Z}{2n\Z}$ on the set $(\frac{\Z}{2\Z})^n=\{+1,-1\}^n$. The generator $\gz$ acts as:
\equ{\gz\bullet (d_1,d_2,\cdots,d_n)=(-d_n,d_1,d_2,\cdots,d_{n-1}).}
Let $0\leq i\leq 2n-1$ where $n=2^lm\in \mbb{N}, 2\nmid m,l\in \mbb{N}\cup\{0\}$. Then \equ{\mid\{d\in \{+1,-1\}^n\mid \gz^i\bullet d=d\}\mid =
	\begin{cases}
		2^n \text{ if }i=0,\\
		2^{gcd(i,n)} \text{ if }i=2^{l+1}j,\\
		0\text{ otherwise}.
	\end{cases}
	}
\end{theorem}
\begin{proof}
Clearly the theorem holds for $n=1,2$. 
	
	For any $n\in \N$, if $i=0$ then clearly 
	\equ{\mid\{d\in \{+1,-1\}^n\mid \gz^i\bullet d=d\}\mid =2^n.}
	Let $n=2^lm,2\nmid m\in \N,l\in \N\cup\{0\}$. Let $i=2^{l+1}j,gcd(j,m)=1$. 
	Suppose $l=0,n=m\geq 3$. Then $(d_1,d_2,\cdots,d_m)\in \{+1,-1\}^m$ is a fixed point of $\gz^2$ if and only if 
	\equ{(-d_{m-1},-d_m,d_1,d_2,\cdots,d_{m-2})=(d_1,d_2,\cdots,d_m).}
	So either $d_i=(-1)^i,1\leq i\leq m$ or $d_i=(-1)^{i+1},1\leq i\leq m$. Hence there are only two fixed points for $\gz^2$. Now \equ{\gz^2\bullet d=d\Ra \gz^{2j}\bullet d=d} and if $j'\in \{1,2,\cdots,m-1\}$ such that $jj'\equiv 1\mod m$ then $2jj'\equiv 2\mod 2m$ and \equ{\gz^{2j}\bullet d=d\Ra (\gz^{2j})^{j'}\bullet d=d\Ra \gz^2\bullet d=d.}
	Hence we have 
	\equ{\mid\{d\in \{+1,-1\}^{n}\mid \gz^{2j}\bullet d=d\}\mid=\mid\{d\in \{+1,-1\}^{n}\mid \gz^{2}\bullet d=d\}\mid=2=2^{gcd(i,n)}.} 
	
	Let $n=2^lm,2\nmid m\in \N,l\in \N\cup\{0\}$. Let $i=2^{l+1}j,gcd(j,m)=1$. Suppose $l\geq 1$. Then consider $n'=\frac n2,i'=\frac i2$. So $i=2i',n=2n'$. We have by induction, using the theorem for values smaller than $l$
	\equ{\mid\{d\in \{+1,-1\}^{n'}\mid \gz^{i'}\bullet d=d\}\mid=2^{gcd(i',n')}=2^{2^{l-1}}.} 
	Let $(d_1,e_1,d_2,e_2,\cdots,d_{n'},e_{n'})\in \{+1,-1\}^n$. We observe that 
	\equ{\gz^{2i'}\bullet (d_1,e_1,d_2,e_2,\cdots,d_{n'},e_{n'})=(d_1,e_1,d_2,e_2,\cdots,d_{n'},e_{n'})}
	if and only if for $<\gt>=\frac{\Z}{2n'\Z},\gt\bullet (d_1,d_2,\cdots,d_{n'})=(-d_{n'},d_1,\cdots,d_{n'-1})$ we have 
	\equ{\gt^{i'}\bullet (d_1,d_2,\cdots,d_{n'})=(d_1,d_2,\cdots,d_{n'}) \text{ and }\gt^{i'}\bullet (e_1,e_2,\cdots,e_{n'})=(e_1,e_2,\cdots,e_{n'}).}
	Hence \equ{\mid\{d\in \{+1,-1\}^n\mid \gz^i\bullet d=d\}\mid=\mid\{d\in \{+1,-1\}^{n'}\mid \gt^{i'}\bullet d=d\}\mid^2=2^{2^l}=2^{gcd(i,n)}.}
	Let $n=2^lm,2\nmid m\in \N,l\in \N\cup\{0\}$. Let $i=2^{l+1}j,gcd(j,m)=t>1$. Then consider $n'=2^l\frac mt,i'=2^{l+1}\frac jt$. So $i=ti',n=tn'$. We have by induction, using the theorem for values smaller than $n$,
	\equ{\mid\{d\in \{+1,-1\}^{n'}\mid \gz^{i'}\bullet d=d\}\mid=2^{gcd(i',n')}=2^{2^l}.} 
	Let $(d_{11},d_{21},\cdots,d_{t1},d_{12},d_{22},\cdots,d_{t2},\cdots,d_{1n'},d_{2n'},\cdots,d_{tn'})\in \{+1,-1\}^{tn'}$. We observe that \equa{\gz^{ti'}\bullet (d_{11},d_{21},&\cdots,d_{t1},d_{12},d_{22},\cdots,d_{t2},\cdots,d_{1n'},d_{2n'},\cdots,d_{tn'})\\&=(d_{11},d_{21},\cdots,d_{t1},d_{12},d_{22},\cdots,d_{t2},\cdots,d_{1n'},d_{2n'},\cdots,d_{tn'})} if and only if \equ{\gt^{i'}\bullet (d_{j1},d_{j2},\cdots,d_{jn'})=(d_{j1},d_{j2},\cdots,d_{jn'})\text{ for }1\leq j\leq t.}
	So we have
	\equ{\mid\{d\in \{+1,-1\}^{n}\mid \gz^{i}\bullet d=d\}\mid=\mid\{d\in \{+1,-1\}^{n'}\mid \gt^{i'}\bullet d=d\}\mid^t=2^{2^lt}=2^{gcd(i,n)}.} 
	
	Now let $n=2^lm,2\nmid m\in \N,l\in \N\cup\{0\}$. Let $i=2^kj,0\leq k\leq l,2\nmid j$.
	Here let $k=0$, that is, $i$ is odd and suppose $gcd(i,n)=1\Ra gcd(i,2n)=1$. If $(d_1,d_2,\cdots,d_n)$ is a fixed point of $\gz$ then we have $(d_1,d_2,\cdots,d_n)=(-d_n,d_1,d_2,$ $\cdots,d_{n-1})\Ra d_1=-d_1$ which is impossible. Hence there are no fixed points of $\gz$. Now let $1\leq i'\leq 2n-1$ such that $ii'\equiv 1 \mod 2n$. Then $\gz^i\bullet d=d\Ra (\gz^i)^{i'}\bullet d=d\Ra \gz\bullet d=d$ which is impossible. Hence there are no fixed points of $\gz^i$. 
	
	Now let $n=2^lm,2\nmid m\in \N,l\in \N\cup\{0\}$. Let $i=2^kj,0\leq k\leq l,2\nmid j$.
	Here again let $k=0$, that is, $i$ is odd and suppose $1<gcd(i,n)=t\Ra gcd(i,2n)=t$. 
	Then we consider $n'=\frac nt,i'=\frac it$ and apply a similar argument as before, to conclude that there are no fixed points for $\gz^i$ as there are no fixed points for $\gt^{i'}$.
	
	Now let $n=2^lm,2\nmid m\in \N,l\in \N\cup\{0\}$. Let $i=2^kj,0\leq k\leq l,2\nmid j$.
	Here let $k\geq 1$. Then we consider $n'=\frac n2,i'=\frac i2$ and apply a similar argument as before to conclude that there are fixed points for $\gz^i$ as there are no fixed points for $\gt^{i'}$.
	Hence Theorem~\ref{theorem:FixPoints} follows.
	
\end{proof}
\begin{theorem}
The cardinality of the orbit space $S_n\bs\mcl{S}^{tnz}_{2n}(\mbb{R})$ is given by 
\equ{\frac1{2n}\us{k\mid n,2\nmid k}{\sum}\gf(k)2^{\frac nk},}
 for $n\geq 2$ where $\gf$ is the Euler-Totient function. 
\end{theorem}
\begin{proof}
First we observe that $O_M$ is conjugate to $O_M^t=-O_M=O_N$ where $N=MP^{-1}$ with $P$ a permutation matrix in $S_n$ such that $Pe^n_i=e^n_{n-i+1}, 1\leq i\leq n$.
So it is enough to count the number of orbits for the action of $S_n$ on the set $\{O_M\mid M\in Mat^{tnz}_{2n}(\mbb{R})\}$ of orientation sign matrices by conjugation.

Now we observe that this action is isomorphic to the restricted action of $P_n$ to $S_n$ on the space of left cosets of $K_n\subs P_n$, that is, the map $\gf:\{pK_n\mid p\in P_n\}\lra \{O_M\mid M\in Mat^{tnz}_{2n}(\mbb{R})\}=\{pO_np^{-1}\mid p\in P_n, O_n$ is the standard orientation sign matrix$\}$ taking $pK_n$ to $pO_np^{-1}$ is a well defined map and an isomorphism of the $P_n$-sets. Also the set of combinatorial representations (a subset of $P_n$) of elements in $Mat_{2n}^{tnz}(\mbb{R})$ which give orientation sign matrix $pO_np^{-1}$ is precisely the right coset $K_np^{-1}$.

The cardinality of the orbit space $S_n\bs\mcl{S}^{tnz}_{2n}(\mbb{R})$ is therefore given by 
\equ{\frac{1}{n!}\bigg(\us{C \text{ a left coset of }K_n}{\sum}\mid \Stab(C)\mid\bigg).}
If $C=pK_n$ then $\Stab(C)=S_n\cap pK_np^{-1}$. Now let us find a system of distinct coset representatives of $K_n$ in $P_n$. A system of complete left coset representatives is precisely given by 
\equ{\{(1,a_2,a_3,\cdots,a_n)\mid (a_2,a_3,\cdots,a_n) \text{ is a signed permutation of }2,3,\cdots,n \}.} This set also has the right cardinality $2^{n-1}(n-1)!=\frac{\mid P_n\mid }{\mid K_n\mid}$.
Let $D_n$ be the subgroup of $P_n$ corresponding to the diagonal matrices in the group of signed permutation matrices. Then we have $D_n\cong (\frac{\Z}{2\Z})^n=\{+1,-1\}^n$ and $P_n\cong S_n \ltimes D_n$ where $S_n$ is the subgroup of $P_n$ corresponding to permutation matrices. If $p=(a_1,a_2,\cdots,a_n)\in P_n$ the matrix associated to $p$ is $P$ where it is given by \equa{P&=(\sgn(a_1)e^n_{\sgn(a_1)a_1},\sgn(a_2)e^n_{\sgn(a_2)a_2},\cdots,\sgn(a_n)e^n_{\sgn(a_n)a_n})\\&=(e^n_{\sgn(a_1)a_1},e^n_{\sgn(a_2)a_2},\cdots,e^n_{\sgn(a_n)a_n}).Diag(\sgn(a_1),\sgn(a_2),\cdots,\sgn(a_n))\\&=QD\text{ where }Q \text{ is a permutation matrix and }D\text{ is a diagonal matrix}.} 
In the group $P_n$ this multiplication is expressed as 
\equa{p&=(a_1,a_2,\cdots,a_n)\\&=(\sgn(a_1)a_1,\sgn(a_2)a_2,\cdots,\sgn(a_n)a_n).(\sgn(a_1)1,\sgn(a_2)2,\cdots,\sgn(a_n)n)\\&=qd\text{ where }q\in S_n,d\in D_n.} 
So $p=qd\Ra \mid\Stab(pK_n)\mid=\mid S_n\cap pK_np^{-1}\mid =\mid p^{-1}S_np\cap K_n\mid =\mid d^{-1}S_nd\cap K_n\mid=\mid dS_nd\cap K_n\mid$ since $d^{-1}=d$.
So we have
\equa{&\frac{1}{n!}\bigg(\us{C \text{ a left coset of }K_n}{\sum}\mid \Stab(C)\mid\bigg)=\\&\frac {(n-1)!}{n!}\bigg(\us{(d_2,d_3,\cdots,d_n)\in (\frac{\Z}{2\Z})^{n-1}}{\sum}\mid (1,2d_2,3d_3,\cdots,nd_n)S_n(1,2d_2,3d_3,\cdots,nd_n)\cap K_n\mid\bigg)\\
	&=\frac 1{2n}\bigg(\us{d\in D_n}{\sum}\mid dS_nd\cap K_n\mid\bigg).} 
Now let $\psi:P_n\lra S_n$ be the surjective map with kernel $D_n$ given by 
\equ{\psi(a_1,a_2,\cdots,a_n)=(\sgn(a_1)a_1,\sgn(a_2)a_2,\cdots,\sgn(a_n)a_n).} 
Then $\psi(K_n)=\langle(n,1,2,\cdots,(n-1)) \rangle$ is a cyclic subgroup order $n$ generated by the $n$-cycle $(123\cdots n)$ in cycle notation. We have $\psi(q)=q$ for $q\in S_n,\psi(d)=\text{ identity for }d\in D_n$.
For $q\in S_n,dqd\in K_n\Ra q=\psi(q)=\psi(dqd)=(n,1,\cdots,(n-1))^i=(n-i+1,n-i+2,\cdots,n-1,n,1,2,\cdots,n-i-1,n-i)$ for some $0\leq i\leq n-1$.

Let $q=(a_1,a_2,\cdots,a_n)\in S_n,d=(1d_1,2d_2,\cdots,nd_n)\in D_n$ where $(d_1,d_2,\cdots,d_n)\in (\Z/2\Z)^n=\{+1,-1\}^n$ then we have 
\equ{dqd=(a_1d_1d_{a_1},a_2d_2d_{a_2},\cdots,a_nd_nd_{a_n}).}
So we have
\equa{&d(n,1,\cdots,(n-1))^id=\\
	&((n-i+1)d_1d_{n-i+1},(n-i+2)d_2d_{n-i+2},\cdots,(n-1)d_{i-1}d_{n-1},nd_id_n,\\&1d_{i+1}d_1,2d_{i+2}d_2,\cdots,(n-i-1)d_{n-1}d_{n-i-1},(n-i)d_nd_{n-i})}
Hence for $0\leq i\leq n-1, d(n,1,\cdots,(n-1))^id\in K_n$ then either the following holds:

\equan{T1}{&d(n,1,\cdots,(n-1))^id=(-n,1,2,\cdots,(n-1))^i\text{ and }\\d_j&=-d_{n-i+j}\text{ for } 1\leq j\leq i\text{ and }d_{i+j}=d_j\text{ for }1\leq j\leq n-i}, 

or the following holds:

\equan{T2}{&d(n,1,\cdots,(n-1))^id=(-n,1,2,\cdots,(n-1))^{n+i}\text{ and }\\d_j&=d_{n-i+j}\text{ for } 1\leq j\leq i\text{ and }d_{i+j}=-d_j\text{ for }1\leq j\leq n-i.} 

Hence we consider the action of the cyclic group $\frac{\Z}{2n\Z}=<\gz>$ on the group $(\Z/2\Z)^n=\{+1,-1\}^n$ as follows. The action of the generator $\gz$ is given by
\equ{\gz\bullet(d_1,d_2,\cdots,d_n)=(-d_n,d_1,d_2,\cdots,d_{n-1}).}

In terms of this group action Equations~\ref{Eq:T1} imply \equ{\gz^i\bullet (d_1,d_2,\cdots,d_n)=(d_1,d_2,\cdots,d_n).}

In terms of this group action Equations~\ref{Eq:T2} imply \equ{\gz^{n+i}\bullet (d_1,d_2,\cdots,d_n)=(d_1,d_2,\cdots,d_n).}
So we have 
\equa{\frac{1}{n!}\bigg(\us{C \text{ a left coset of }K_n}{\sum}\mid \Stab(C)\mid\bigg)
	&=\frac 1{2n}\bigg(\us{d\in D_n}{\sum}\mid dS_nd\cap K_n\mid\bigg)\\
	&=\frac1{2n}\bigg(\us{d\in \{+1,-1\}^n}{\sum}\mid \{i\mid \gz^i\bullet d=d,0\leq i\leq 2n-1\}\mid\bigg)\\
&=\frac 1{2n}\bigg(\us{i=0}{\os{2n-1}{\sum}}\mid \{d\in \{+1,-1\}^n\mid \gz^i\bullet d=d\}\mid\bigg).}

We prove the theorem using Theorem~\ref{theorem:FixPoints}. Now for an integer $j$ we have $1\leq j\leq m-1 \Llra 1\leq 2^{l+1}j< 2^{l+1}m=2n$. Let $k\mid m,k\neq 1$. The cardinality of the set $\{j\mid 1\leq j\leq m-1, gcd(j,m)=\frac mk\}$ is exactly $\gf(k)$ where $\gf$ is the Euler-totient function. Hence the cardinality of the set $\{i\mid 1\leq i\leq 2n-1,gcd(i,n)=\frac nk,i=2^{l+1}j\}$ is $\gf(k)$. So we have

\equa{\frac{1}{n!}\bigg(\us{C \text{ a left coset of }K_n}{\sum}\mid \Stab(C)\mid\bigg)
	&=\frac 1{2n}\bigg(\us{i=0}{\os{2n-1}{\sum}}\mid \{d\in \{+1,-1\}^n\mid \gz^i\bullet d=d\}\mid\bigg)\\&=\frac 1{2n}\bigg(2^n+\us{k\mid m,k\neq 1}{\sum}\gf(k)2^{\frac nk}\bigg)\\&=\frac1{2n}\us{k\mid n,2\nmid n}{\sum}\gf(k)2^{\frac nk}.}

This completes the proof of the theorem.
\end{proof}
\begin{remark}
The initial values for $2\leq n\leq 10$ are given as $1,2,2,4,6,10,16,30,52$. Also refer OEIS Sloane sequence A000016~\cite{A000016}.
\end{remark}
\section{\bf{Open Questions: Enumeration of the Generic Point Arrangements, the Antipodal Point Arrangements in Higher Dimensions}}
Theorem~\ref{theorem:Enumeration} gives parameter spaces for the isomorphism classes of the generic point arrangements of cardinality $n$ and the isomorphism classes of the antipodal point arrangements of cardinality $2n$ in the Euclidean space $\mbb{R}^m$. In Sections~\ref{sec:CombRepGPA},~\ref{sec:CombRepElement}, we have obtained combinatorial representations of a generic point arrangement of cardinality $n$ in the plane and of an element in $Mat^{tnz}_{2n}(\mbb{R})$ which lead to the enumeration of parameter spaces for $m=2$. We also have observed that there is a single isomorphism class for the antipodal point arrangements of cardinality $n$ in the plane in Theorem~\ref{theorem:Trasitive}.

The analogous enumeration questions about the cardinalities of these parameter spaces  $((\mbb{R}^*)^n\rtimes S_n)\bs\mcl{S}^{tnz}_{mn}(\mbb{R})$ and $S_n\bs\mcl{S}^{tnz}_{mn}(\mbb{R})$ are still open for $m\geq 3$. The enumeration of the space $\mcl{S}^{tnz}_{mn}(\mbb{R})$ of strata for the totally nonzero Grassmannian is also open for $m\geq 3$. It is known that there is more than one isomorphism class for the antipodal point arrangements in the space $\mbb{R}^3$ of cardinality $2n$ for $n=6$ (see C.~P.~Anil Kumar~\cite{IJPA1},~\cite{IJPA2}) even though it is not a difficult exercise to show that there is a single isomorphism class of the antipodal point arrangements in the space $\mbb{R}^3$ of cardinality $2n$ for each $n=3,4,5$. The antipodal point arrangements in general are combinatorially classified in~\cite{IJPA2}. 

\end{document}